\documentclass[a4paper]{amsart}
\usepackage{amssymb}

\input xy
\xyoption{all}

\newcommand{\Id}{\mathrm{Id}}
\newcommand{\Comp}{\mathsf{Comp}}

\newcommand{\Cl}{\mathrm{Cl}}
\newcommand{\lin}{\mathrm{lin}}
\newcommand{\Int}{\mathrm{Int}}

\newcommand{\A}{\mathcal A}

\newcommand{\K}{\mathcal K}

\newcommand{\B}{\mathcal B}

\newcommand{\R}{\mathbb R}
\newcommand{\N}{\mathbb N}
\newcommand{\T}{\mathcal T}
\newcommand{\F}{\mathcal F}
\newcommand{\C}{\mathcal C}
\newcommand{\M}{\mathbb M}
\newcommand{\E}{\mathbb E}
\newcommand{\BK}{\mathcal {BK}}
\newcommand{\BT}{\mathcal {BT}}

\newtheorem{theorem}{Theorem}
\newtheorem{df}{Definition}
\newtheorem{lemma}{Lemma}

\begin{document}

\title{On balanced capacities}
\author{Taras Radul}

\maketitle

Institute of Mathematics, Casimirus the Great University of Bydgoszcz, Poland;
\newline
Department of Mechanics and Mathematics, Ivan Franko National University of Lviv,
Universytettska st., 1. 79000 Lviv, Ukraine.
\newline
e-mail: tarasradul@yahoo.co.uk



\begin{abstract} We consider  capacity (fuzzy measure, non-additive probability) on a compactum as a monotone  cooperative normed game. Then it is naturally to consider  probability measures as elements of core of such game.
We prove an analogue of Bondareva-Shapley theorem that non-emptiness of the core is equivalent to balancedness of the capacity. We investigate categorical properties of balanced capacities and give characterizations of some fuzzy integrals of  balanced capacities.
\end{abstract}

\maketitle

\section{Introduction}
Capacities (non-additive measures, fuzzy measures) were introduced by Choquet in \cite{Ch} as a natural generalization of additive measures. They found numerous applications (see for example \cite{EK},\cite{Gil},\cite{Sch}). Categorical and topological properties of spaces of upper-semicontinuous normed capacities on compact Hausdorff spaces were investigated in \cite{NZ}. In particular, there was built the capacity functor which is a functorial part of a capacity monad $\M$.

Capacities can be considered as monotone cooperative games. The important solution concept of a cooperative games is the core notion \cite{Gilies}. The core consists of additive set functions which dominate the characteristic function of the game. The Bondareva-Shapley Theorem (\cite{Bond} and \cite{Shap}) provides a characterization of games on finite set with non-empty core. It states that the core is non-empty if and only if the game is balanced.
The games on infinite sets were considered in \cite{Sch1}, \cite{Kan}, \cite{Kan1}, \cite{Pin}, \cite{BarPin}, where the core consists as of additive set functions so and of $\sigma$-additive set functions.

We consider in this paper a cooperative game on a compactum with the characteristic function represented by a  upper-semicontinuous normed capacity. It seems  natural to consider the core consisting of $\sigma$-additive regular normed measures (probability measures). We prove  an analogue of Bondareva-Shapley theorem that non-emptiness of the core is equivalent to balancedness of the capacity. We consider categorical properties of the space of balanced capacities $MBX$ in Section 2. We show that  $MB$ is a subfunctor of the capacity functor $M$ but $MB$ is not a submonad of the capacity monad. Finally, we give a characterization of some fuzzy integrals generated by balanced capacities in Section 3.

\section{Balanced capacities} By compactum we mean a compact Hausdorff space.  In what follows, all spaces are assumed to be compacta except for $\R$ and maps are assumed to be continuous. Let $A$ be a subset of $X$. By $\Cl A$ we denote the closure of $A$ in $X$. By $\F(X)$ we denote the family of all closed subsets of $X$.

For each compactum $X$ we denote by $C(X)$ the Banach space of all
continuous functions $\phi:X\to\R$ with the usual $\sup$-norm: $
\|\phi\| =\sup\{|\phi(x)|\mid x\in X\}$. We also consider on $C(X)$ the natural partial order. For $\lambda\in\R$ we denote by $\lambda_X$ the constant function on $X$ with values $\lambda$.   By $C^+(X)$ we denote the set of all functions from $C(X)$ which are  $\ge 0_X$.

We need the definition of capacity on a compactum $X$. We follow a terminology of \cite{NZ}.
A function $\nu:\F(X)\to [0,1]$  is called an {\it upper-semicontinuous capacity} on $X$ if the three following properties hold for each closed subsets $F$ and $G$ of $X$:

1. $\nu(X)=1$, $\nu(\emptyset)=0$,

2. if $F\subset G$, then $\nu(F)\le \nu(G)$,

3. if $\nu(F)<a$, then there exists an open set $O\supset F$ such that $\nu(B)<a$ for each compactum $B\subset O$.

If $F$ is a one-point set we use a simpler notation $\nu(a)$ instead $\nu(\{a\})$.
A capacity $\nu$ is extended in \cite{NZ} to all open subsets $U\subset X$ by the formula $\nu(U)=\sup\{\nu(K)\mid K$ is a closed subset of $X$ such that $K\subset U\}$.

It was proved in \cite{NZ} that the space $MX$ of all upper-semicontinuous  capacities on a compactum $X$ is a compactum as well, if a topology on $MX$ is defined by a subbase that consists of all sets of the form $O_-(F,a)=\{c\in MX\mid c(F)<a\}$, where $F$ is a closed subset of $X$, $a\in [0,1]$, and $O_+(U,a)=\{c\in MX\mid c(U)>a\}$, where $U$ is an open subset of $X$, $a\in [0,1]$. Since all capacities we consider here are upper-semicontinuous, in the following we call elements of $MX$ simply capacities.

The following definition is a topological analogue of the definition of balanced game given in \cite{Sch1}.

\begin{df}\label{bal} A capacity $\nu\in MX$ is called balanced if we have $\sum_{i=1}^n\lambda_i\nu(A_i)\le 1$ for each numbers $\lambda_1,\dots,\lambda_n\in\R$ and closed subsets $A_1,\dots,A_n$ of $X$ such that $\sum_{i=1}^n\lambda_i\chi_{A_i}\le 1_X$ where $\chi_{A_i}$ is the characteristic function of the set $A_i$.
\end{df}

\begin{lemma}\label{roz} For each closed subsets $A_1,\dots,A_n$ of a compactum $X$ such that   $\sum_{i=1}^n\lambda_i\chi_{A_i}\le 1_X$ there exist closed subsets  $B_1,\dots,B_n$ of $X$ such that $\sum_{i=1}^n\lambda_i\chi_{B_i}\le 1_X$ and $A_i\subset \Int B_i$.
\end{lemma}


By probability measure on a compactum $X$ we mean normed $\sigma$-additive regular Borel measure. By Riesz Theorem we can alternatively describe a probability measure as linear  positively defined functional
$\mu:C(X)\to\R$ with norm $1$. The correspondence could by described as follows $\mu(A)=\inf \{\mu(\varphi)\mid \varphi\in C(X) $ with $\varphi\ge\chi_{A}\}$ for a closed subset $A$ of $X$. Let us remark that linear positively defined functional $\mu$ has norm $1$ if and only if $\mu(1_X)=1$. One can find more information about probability measures on compacta in \cite{Fed}. We identify representation of probability measures as set functions and as functionals in the following.

\begin{df}\label{core}  Let $\nu\in MX$ be a capacity. We say that a probability measure $\mu$ belongs to the core of $\nu$ if $\mu(A)\ge\nu(A)$ for each closed subset $A$ of $X$.
\end{df}

The following theorem is a topological analogue of the definition of Bondareva-Shapley Theorem. We use in the proof the idea of the proof of Han-Banach Theorem which was also used in \cite{Sch1}.

\begin{theorem} The core of a capacity $\nu$ is non-empty if and only if $\nu$ is balanced.
\end{theorem}

\begin{proof} Necessity.  Consider any probability measure  $\mu$ belonging to the core of $\nu$. Take any numbers $\lambda_1,\dots,\lambda_n\in\R$ and closed subsets $A_1,\dots,A_n$ of $X$ such that $\sum_{i=1}^n\lambda_i\chi_{A_i}\le 1_X$. By Lemma \ref{roz} we can choose closed subsets  $B_1,\dots,B_n$ of $X$ such that $\sum_{i=1}^n\lambda_i\chi_{B_i}\le 1_X$ and $A_i\subset \Int B_i$. For each $i\in\{1,\dots,n\}$ choose a continuous function $\varphi_i:X\to\R$ such that $\varphi_i(a)=1$ for each $a\in A_i$ and $\varphi_i(b)=0$ for each $b\in X\setminus \Int B_i$.   Then we have $\sum_{i=1}^n\lambda_i\nu(A_i)\le\sum_{i=1}^n\lambda_i\mu(A_i)
\le\sum_{i=1}^n\lambda_i\mu(\varphi_i)=\mu(\sum_{i=1}^n\lambda_i\varphi_i)$. Since $\sum_{i=1}^n\lambda_i\varphi_i\le\sum_{i=1}^n\lambda_i\chi_{B_i}\le 1_X$, we have $\mu(\sum_{i=1}^n\lambda_i\varphi_i)\le  1$.

Sufficiency. Let $\nu$ be a balanced capacity. Define a functional $p:C(X,[0,1))\to [0,+\infty)$ by the formula $p(f)=\sup\{\sum_{i=1}^n\lambda_i\nu(A_i)\mid n\in\N$, $A_i$ are closed subsets of $X$ and $\lambda_i$ are non-negative numbers such that $\sum_{i=1}^n\lambda_i\chi_{A_i}\le 1_X\}$. It is easy to see that $p$ is superadditive, $p(\lambda\phi)=\lambda p(\phi)$ for each $\lambda\ge 0$, $\phi\in C(X,[0,1))$  and $p(1_X)=1$.

Consider any linear subspace $Y\subset C(X)$ such that $Y$ contains all constants. Suppose we have built a linear functional $\mu:Y\to \R$ such that $\mu(1_X)=1$ and $\mu(\varphi)\ge  p(\varphi)$ for each $\varphi\in Y\cap C(X,[0,1))$. Consider any $\phi\in C(X)\setminus Y$.  We can  extend $\mu$ to a linear functional $\mu'$ defined on the linear space $\lin (\{\phi\}\cup Y)=\{\lambda\phi+y\mid y\in Y$ and $\lambda\in\R\}$ putting $\mu'(\lambda\phi+y)=  \lambda c+\mu(y)$.

We need to show that $c$ may be chosen so that   $\lambda c+\mu(y)\ge p(\lambda\phi+y)$ for each $\lambda\phi+y)\in C(X,[0,1))$. Consider any $\lambda>0$. Then $p(\lambda\phi+y)=\lambda p(\phi+y/\lambda)$. Put $y/\lambda=t\in Y$. If $\lambda<0$, we have $p(\lambda\phi+y)=-\lambda p(-\phi-y/\lambda)$ and we put $-y/\lambda=z\in Y$.

It is enough to choose $c\in\R$ so that for each $z,t\in Y$ with $t+\phi\ge 0$ and $z-\phi\ge 0$ the following two inequalities are satisfied $c+\mu(t)\ge p(t+\phi)$ and $-c+\mu(z)\ge p(z-\phi)$. It can be rewritten as follows $p(t+\phi)-\mu(t)\le c\le\mu(z)- p(z-\phi)$.

Existence of such $c$ follows from following inequalities $p(t+\phi)+p(z-\phi)\le p(t+z)\le\mu(t+z)\le\mu(t)+\mu(z)$. Let us remark that the inequality $p(t+\phi)+p(z-\phi)\le p(t+z)$ follows from the superadditivity of $p$ and the inequality $p(t+z)\le\mu(t+z)$ follows from $t,z\in Y$.

Zorn Lemma implies that there exists a linear functional $\mu:C(X)\to\R$ such that $\mu(1_X)=1$ and $\mu(\varphi)\ge  p(\varphi)$ for each $\varphi\in  C(X,[0,1))$. The last property implies that $\mu$ is positively defined, hence $\mu$ is a probability measure.
\end{proof}

\section{Categorical properties of balanced capacities}

By $\Comp$ we denote the category of compact Hausdorff
spaces (compacta) and continuous maps.
 For a continuous map of compacta $f:X\to Y$ we define the map $Mf:MX\to MY$ by the formula $Mf(\nu)(A)=\nu(f^{-1}(A))$ where $\nu\in MX$ and $A\in\F(Y)$. The map $Mf$ is continuous.  In fact, this extension of the construction $M$ defines the capacity functor $M$ in the category $\Comp$ (see \cite{NZ} for more detailes).

We recall the notion is the notion of monad (or triple) in the sense of S.Eilenberg and J.Moore \cite{EM}.  We define them only for the category $\Comp$.

A {\it monad} \cite{EM} $\E=(E,\eta,\mu)$ in the category $\Comp$ consists of an endofunctor $E:{\Comp}\to{\Comp}$ and natural transformations $\eta:\Id_{\Comp}\to F$ (unity), $\mu:F^2\to F$ (multiplication) satisfying the relations $\mu\circ E\eta=\mu\circ\eta E=${\bf 1}$_E$ and $\mu\circ\mu E=\mu\circ E\mu$. (By $\Id_{\Comp}$ we denote the identity functor on the category ${\Comp}$ and $E^2$ is the superposition $E\circ E$ of $E$.)

The functor $M$ was completed to the monad $\M=(M,\eta,\mu)$ \cite{NZ}, where the components of the  natural transformations are defined as follows: $\eta X(x)(F)=1$ if $x\in F$ and $\eta X(x)(F)=0$ if $x\notin F$;
$\mu X(\C)(F)=\sup\{t\in[0,1]\mid \C(\{c\in MX\mid c(F)\ge t\})\ge t\}$, where $x\in X$, $F$ is a closed subset of $X$ and $\C\in M^2(X)$.

For a compactum $X$ we denote by $MBX$ the subspace of $MX$ consisting of balanced capacities. The aim of this section is to show that $MB$ forms subfunctor of the functor $M$ but is not a submonad.

\begin{lemma}\label{closed} The set $MBX$ is closed in $MX$.
\end{lemma}

\begin{proof} Take any $\nu\in MX\setminus MBX$. Then there exist positive numbers $\lambda_1,\dots,\lambda_n\in\R$ and closed subsets $A_1,\dots,A_n$ of $X$ such that $\sum_{i=1}^n\lambda_i\chi_{A_i}\le 1_X$ and $\sum_{i=1}^n\lambda_i\nu(A_i)=a> 1$. Put $\delta_i=\frac{a-1}{n\lambda_i}>0$ for $i\in\{1,\dots,n\}$. By Lemma \ref{roz} we can choose for each $i\in\{1,\dots,n\}$ an  open set $O_i$ such that $\sum_{i=1}^n\lambda_i\chi_{O_i}\le 1_X$ and $A_i\subset O_i$.

Put $V_i=\{\mu\in MX\mid \mu(O_i)>\nu(A_i)-\delta_i\}$ and $V=\cap_{i=1}^n V_i$. Then $V$ is an open subset of $MX$ containing $\nu$. Consider any $\mu\in V$. Then, for each $i\in\{1,\dots,n\}$, there exists a closed subset $B_i$ of $X$ such that  $B_i\subset O_i$ and $\mu(B_i)>\nu(A_i)-\delta_i$. Then we have $\sum_{i=1}^n\lambda_i\chi_{B_i}\le 1_X$ and $\sum_{i=1}^n\lambda_i\nu(B_i)> 1$. Hence $\mu\notin MBX$.
\end{proof}

\begin{lemma}\label{sub} Let $f:X\to Y$ be a continuous map between compacta $X$ and $Y$. Then we have  $Mf(MBX)\subset MBY$.
\end{lemma}

\begin{proof} Take any $\nu\in Mf(MBX)$. Then there exists $\mu\in MBX$ such that $Mf(\mu)=\nu$. Take any positive numbers $\lambda_1,\dots,\lambda_n\in\R$ and closed subsets $A_1,\dots,A_n$ of $Y$ such that $\sum_{i=1}^n\lambda_i\chi_{A_i}\le 1_Y$. Put $B_i=f^{-1}(A_i)$. We have $\sum_{i=1}^n\lambda_i\chi_{B_i}(x)=\sum_{i=1}^n\lambda_i\chi_{A_i}(f(x))$ for each $x\in X$, hence $\sum_{i=1}^n\lambda_i\chi_{B_i}\le 1_X$. Thus $\sum_{i=1}^n\lambda_i\nu(A_i)=\sum_{i=1}^n\lambda_i\mu(B_i)\le 1$ and $\nu\in MBY$.
\end{proof}

Lemmas \ref{closed} and \ref{sub} imply that $MB$ is a subfunctor of $M$.
However, we will show that  $MB$ does not form a submonad of $\M$.

Consider the compactum $X=\{1,2,3,4,5,6\}$ and denote $A_1=\{1,2,3\}$, $A_2=\{1,4,5\}$, $A_3=\{2,5,6\}$, $A_4=\{3,4,6\}$. Let $\nu_0\in MX$ be a minimal capacity such that  $\nu_0(A_j)=2/3$ for each $j\in\{1,2,3,4\}$. Since $|\{j\in \{1,2,3,4\}\mid x\in A_j\}|=2$ for each $x\in X$, we have  $\sum_{i=1}^41/2\chi_{A_i}= 1_X$. We also have $\sum_{i=1}^41/2\nu_0(A_i)=4/3>1$, hence $\nu_0\notin MBX$.

Denote $B^{t}=\{\nu\in MBX\mid \nu(B)\ge t\}$ for each closed subset $B\subset X$ and $t\in[0,1]$. Let $\A\in M(MB(X))$ be a minimal capacity such that $\A(A_i^{2/3})=2/3$ for each $i\in\{1,2,3,4\}$.

\begin{lemma}\label{od} $\mu_X(\A)=\nu_0$.
\end{lemma}

\begin{proof} We have evidently $\mu_X(\A)(X)=1$.

Consider any proper closed  subset $B\subset X$ such that $B\supset A_i$ for some $i\in\{1,2,3,4\}$. We have $\mu_X(\A)(B)=\sup\{t\in[0,1]\mid \A(B^{t})\ge t\}$. If $t=2/3$, we obtain $\A(B^{2/3})\ge\A(A_i^{2/3})=2/3$. Consider $t>2/3$. Take any $x_0\in X\setminus B$. Then $MBX\setminus B^{t}$ contains Dirac measure  $\delta_{x_0}$, thus $B^t\neq MBX$ and $\A(B^{t})\le 2/3<t$. Hence  $\mu_X(\A)(B)=2/3$.

Now, consider a closed subset $C\subset X$ such that there exists $x_i\in A_i\setminus C$ for each $i\in\{1,2,3,4\}$. Take any $t>0$. Then $\delta_{x_i}\in A_i^{2/3}\setminus C^t$ for each $i\in\{1,2,3,4\}$. Hence $\A(C^t)=0$ and $\mu_X(\A)(C)=0$.
\end{proof}

Take any $l\in\{1,2,3,4\}$ and put $K_l=\{1,2,3,4\}\setminus \{l\}$. Let $\nu_l\in MX$ be a minimal capacity such that  $\nu_l(A_j)=2/3$ for each $j\in K_l$.

\begin{lemma}\label{dwa} The capacity $\nu_l$ is balanced.
\end{lemma}

\begin{proof}

 Consider any positive numbers $\beta_1,\dots,\beta_s\in\R$ and closed subsets $C_1,\dots,C_s$ of $X$ such that $\sum_{i=1}^s\beta_i\chi_{C_i}\le 1_{X}$. Represent $\{1,\dots,s\}=M\bigsqcup P\bigsqcup T$ where $M=\{i\in\{1,\dots,s\}\mid C_i=X\}$, $P=\{i\in\{1,\dots,s\}\mid C_i\neq X$ and there exists $j\in K_l$ such that $C_i\supset A_j\}$ and $T=\{i\in\{1,\dots,s\}\mid C_i\nsupseteq A_j$ for each $j\in K_l\}$. Then we have  $\nu_l(C_i)=0$ for each $i\in T$ and $\nu_l(C_i)=1$ for each $i\in M$.

 We can choose $a\in [0,1]$ such that $\sum_{i\in M}\beta_i\chi_{C_i}= a_{X}$. Put $b=1-a\in [0,1]$.
 For each $i\in P$ we choose a number $j(i)\in K_l$ such that $C_i\supset A_{j(i)}$ and denote $P_j=\{i\in P\mid j(i)=j\}$. Put $\varsigma_j=\sum_{i\in P_j}\beta_i$ for $j\in L$. If $M_j=\emptyset$, we set $\varsigma_j=0$.  Then we have $\sum_{j\in K_l}\varsigma_j\chi_{A_j}\le\sum_{i\in M}\beta_i\chi_{C_i}\le b_{X}$.

 Since $A_s\cap A_j\ne\emptyset$ for each $s$, $j\in K_l$ with $s\neq j$, we have $\varsigma_s+\varsigma_j\le b$. Taking sum on all pairs $(j;s)$ with $s\neq j$, we obtain  $2\sum_{j\in K_l}\varsigma_j\le 3b$. Then we have  $\sum_{i=1}^s\beta_i\nu_l(C_i)=\sum_{i\in M}\beta_i\nu_l(C_i)+\sum_{i\in P}\beta_i\nu_l(C_i)=a+\sum_{j\in K_l}\varsigma_j\cdot2/3\le a+2/3\cdot3/2b=1$.
\end{proof}

\begin{lemma}\label{try} The capacity $\A$ is balanced.
\end{lemma}

\begin{proof}  Take any positive numbers $\lambda_1,\dots,\lambda_k\in\R$ and closed subsets $B_1,\dots,B_k$ of $MBX$ such that $\sum_{i=1}^k\lambda_i\chi_{B_i}\le 1_{MBX}$. Represent $\{1,\dots,k\}=L\bigsqcup M\bigsqcup S$ where $L=\{i\in\{1,\dots,k\}\mid B_i=MBX\}$, $M=\{i\in\{1,\dots,k\}\mid B_i\neq MBX$ and there exists $j\in\{1,\dots,4\}$ such that $B_i\supset A_j^{2/3}\}$ and $S=\{i\in\{1,\dots,k\}\mid B_i\nsupseteq A_j^{2/3}$ for each $j\in\{1,\dots,4\}\}$. Then we have  $\A(B_i)=0$ for each $i\in S$ and $\A(B_i)=1$ for each $i\in L$.

We can choose $a\in [0,1]$ such that $\sum_{i\in L}\lambda_i\chi_{B_i}= a_{X}$. Put $b=1-a\in [0,1]$.
For each $i\in M$ we choose a number $j(i)\in\{1,2,3,4\}$ such that $B_i\supset A_{j(i)}^{2/3}$ and denote $M_j=\{i\in M\mid j(i)=j\}$. Put $\eta_j=\sum_{i\in M_j}\lambda_i$ for $j\in\{1,2,3,4\}$. If $M_j=\emptyset$, we set $\eta_j=0$.  Then we have $\sum_{j=1}^4\eta_j\chi_{A_j^{2/3}}\le\sum_{i\in M}\lambda_i\chi_{B_i}\le b_{MBX}$.

 Since $\nu_l\in\cap_{j\in K_l}A_{j}^{2/3}$ for each $l\in\{1,2,3,4\}$ by Lemma \ref{dwa}, we have $\sum_{j\in K_l}\eta_j\le b$. Taking sum over all  $l\in\{1,2,3,4\}$, we obtain  $3\sum_{j\in K_l}\eta_j\le 4b$. Then we have  $\sum_{i=1}^k\lambda_i\A(B_i)=\sum_{i\in L}\lambda_i\A(B_i)+\sum_{i\in M}\lambda_i\A(B_i)=a+2/3\sum_{i\in M}\lambda_i=a+2/3\sum_{j\in \{1,2,3,4\}}\eta_j\le a+2/3\cdot4/3b\le1$.

\end{proof}

So, we have $\A\in MB(MBX)$ and $\mu X(\A)\notin MBX$. Hence $MB$ is not a submonad of the monad $\M$.

\section{Characterizing of fuzzy integrals generated by balanced capacities}\label{equil}

We characterize the Choquet integral and the t-normed integrals with respect to balanced capacities in this section.

Denote $\varphi_t=\varphi^{-1}([t,+\infty))$ for each $\varphi\in C(X,[0,+\infty))$ and $t\in[0,+\infty)$.

 We consider for a compactum $X$ and  for a  function $f\in  C(X,[0,+\infty))$ an integral respect a capacity $\mu\in MX$ defined by the formula
$$\int_X^{Ch} fd\mu=\int_0^\infty\mu(f_t)dt$$   and call it the Choquet integral \cite{Ch}.

Remind that triangular norm $\ast$ is a binary operation on the closed unit interval $[0,1]$ which is associative, commutative, monotone and $s\ast 1=s$ for each  $s\in [0,1]$ \cite{PRP}. We consider only continuous t-norms in this paper. Integrals generated  by  t-norms are called t-normed integrals and were studied in \cite{We1}, \cite{We2} and \cite{Sua}. So, for a continuous t-norm $\ast$, a capacity $\mu$ and a  function $f\in  C(X,[0,1])$ the corresponding t-normed integral is defined by the formula $$\int_X^{\vee\ast} fd\mu=\max\{\mu(f_t)\ast t\mid t\in[0,1]\}.$$

Let us consider characterizations of Choquet and t-normed integrals. Let $X$ be a compactum.  We call two functions $\varphi$, $\psi\in C(X)$ comonotone (or equiordered) if $(\varphi(x_1)-\varphi(x_2))\cdot(\psi(x_1)-\psi(x_2))\ge 0$ for each $x_1$, $x_2\in X$ \cite{M}.

We denote for a compactum $X$ by $\K(X)$ the set of functionals $\mu:C(X,[0,+\infty))\to[0,+\infty)$ which satisfy the conditions:

\begin{enumerate}
\item $\mu(1_X)=1$;
\item $\mu(\varphi)\le\mu(\psi)$ for each functions $\varphi$, $\psi\in C(X,[0,+\infty))$ such that $\varphi\le\psi$;
\item $\mu(\psi+\varphi)=\mu(\psi)+\mu(\varphi)$ for each comonotone functions $\varphi$, $\psi\in C(X,[0,+\infty))$;
\item $\mu(c\varphi)=c\mu(\varphi)$ for each $c\in[0,+\infty)$ and $\varphi\in C(X,[0,+\infty))$.

\end{enumerate}

It was proved in \cite{CB} for finite compacta and in \cite{Lin} for a general case  that  a   functional $\mu$ on  $C(X,[0,+\infty))$ belongs to $\K(X)$ if and only if there exists a unique capacity $\nu$ such that $\mu$ is the Choquet integral with respect to $\nu$.

Let $\ast$ be continuous t-norm. We denote for a compactum $X$ by $\T(X)$ the set of functionals $\mu:C(X,[0,1])\to[0,1]$ which satisfy the conditions:

\begin{enumerate}
\item $\mu(1_X)=1$;
\item $\mu(\varphi)\le\mu(\psi)$ for each functions $\varphi$, $\psi\in C(X,[0,1])$ such that $\varphi\le\psi$;
\item $\mu(\psi\vee\varphi)=\mu(\psi)\vee\mu(\varphi)$ for each comonotone functions $\varphi$, $\psi\in C(X,[0,1])$;
\item $\mu(c_X\ast\varphi)=c\ast\mu(\varphi)$ for each $c\in\R$ and $\varphi\in C(X,[0,1])$.

\end{enumerate}

It was proved in \cite{CLM} for finite compacta and in \cite{Rad} for a general case  that  a   functional $\mu$ on  $C(X,[0,1])$ belongs to $\T(X)$ if and only if there exists a unique capacity $\nu$ such that $\mu$ is the t-normed integral with respect to $\nu$.

\begin{df} We call a functional $\mu\in\K(X)$ ($\mu\in\T(X)$) balanced if  for any numbers $\lambda_1,\dots,\lambda_n\in\R$ and closed subsets $A_1,\dots,A_n$ of $X$ such that $\sum_{i=1}^n\lambda_i\chi_{A_i}\le 1_X$ there exist functions $\varphi_1,\dots,\varphi_n\in C(X,[0,+\infty))(C(X,[0,1]))$ such that $\chi_{A_i}\le\varphi_i$ for each $i\in\{1,\dots,n\}$ and $\sum_{i=1}^n\lambda_i\mu(\varphi_i)\le 1$.
\end{df}

We denote by $\BK(X)$ ($\BT(X)$) the set of all balanced functional in $\K(X)$ ($\T(X)$).

\begin{theorem} Let $\mu\in \K(X)$. Then $\mu\in \BK(X)$ if and only if there exists $\nu\in MB(X)$ such that  $\mu(\varphi)=\int_X^{Ch} \varphi d\nu$ $\varphi\in C(X,[0,+\infty))$.
\end{theorem}

\begin{proof} Necessity.  We can choose any $\nu\in M(X)$ such that  $\mu(\varphi)=\int_X^{Ch} \varphi d\nu$ by the above mentioned characterization of the Choquet integral from \cite{Lin}. Moreover, we have $\nu(A)=\inf \{\mu(\varphi)\mid \varphi\in C(X,[0,+\infty))$ with $\varphi\ge\chi_{A}\}$ for each closed subset $A$ of $X$ \cite{Lin}. We have to show that $\nu\in MB(X)$.

Consider any numbers $\lambda_1,\dots,\lambda_n\in\R$ and closed subsets $A_1,\dots,A_n$ of $X$ such that $\sum_{i=1}^n\lambda_i\chi_{A_i}\le 1_X$. Since $\mu\in \BK(X)$, we can choose functions $\varphi_1,\dots,\varphi_n\in C(X,[0,+\infty))$ such that $\chi_{A_i}\le\varphi_i$ for each $i\in\{1,\dots,n\}$ and $\sum_{i=1}^n\lambda_i\mu(\varphi_i)\le 1$. Then we have $\sum_{i=1}^n\lambda_i\nu(A_i)=\sum_{i=1}^n\lambda_i\inf \{\mu(\varphi)\mid \varphi\ge\chi_{A_i}\}\le\sum_{i=1}^n\lambda_i\mu(\varphi_i)\le 1$.

Sufficiency. Let $\nu\in MB(X)$ such that  $\mu(\varphi)=\int_X^{Ch} \varphi d\nu$. We have to show that $\mu\in \BK(X)$. Consider any numbers $\lambda_1,\dots,\lambda_n\in\R$ and closed subsets $A_1,\dots,A_n$ of $X$ such that $\sum_{i=1}^n\lambda_i\chi_{A_i}\le 1_X$. By Lemma \ref{roz} we can choose closed subsets  $B_1,\dots,B_n$ of $X$ such that $\sum_{i=1}^n\lambda_i\chi_{B_i}\le 1_X$ and $A_i\subset \Int B_i$. Choose functions $\varphi_1,\dots,\varphi_n\in C(X,[0,+\infty))$ such that $\chi_{A_i}\le\varphi_i\le\chi_{B_i}$ for each $i\in\{1,\dots,n\}$. Then we have $\sum_{i=1}^n\lambda_i\mu(\varphi_i)=\sum_{i=1}^n\lambda_i\int_0^\infty\nu((\varphi_i)_t)dt\le\sum_{i=1}^n\lambda_i\nu(B_i)\le 1$.
\end{proof}

Analogously we can prove a characterization theorem for t-normed integrals.

\begin{theorem} Let $\mu\in \T(X)$. Then $\mu\in \BT(X)$ if and only if there exists $\nu\in MB(X)$ such that  $\mu(\varphi)=\int_X^{\vee\ast} \varphi d\nu$ for each $\varphi\in C(X,[0,1])$.
\end{theorem}

\end{document}